\documentclass[11pt]{amsart}
\usepackage{graphpap}
\usepackage{amssymb, color}
\usepackage{amsmath}
\usepackage{amsfonts}
\usepackage{hyperref}
\newtheorem{thm}{Theorem}
\newtheorem{lem}{Lemma}
\newtheorem{pro}{Proposition}

\theoremstyle{definition} %style des newtheorem qui suivent 
\newtheorem{rem}{Remark}

\catcode`@=11
\def\section{\@startsection{section}{1}%
  \z@{1.5\linespacing\@plus\linespacing}{.5\linespacing}%
  {\normalfont\bfseries\large\centering}}
\catcode`@=12

\def\R{\mathbb R}
\def\C{\mathbb C}
\def\Z{\mathbb Z}

\def\N{\mathbb N}

\def\var{\varepsilon}
\def\lap{\Delta}
\def\be{\begin{equation*}}
\def\ee{\end{equation*}}
\def\beq{\begin{equation}}
\def\eeq{\end{equation}}

\begin{document}

\title[A lower bound on the blow up rate for DS on the torus]{A lower bound on the blow up rate for the Davey-Stewartson system on the torus}
\author[Nicolas Godet]{Nicolas Godet}
\address{University of Cergy-Pontoise, Department of Mathematics, CNRS, UMR 8088, F-95000 Cergy-Pontoise}
\email{nicolas.godet@u-cergy.fr}

\begin{abstract}
We consider the hyperbolic-elliptic version of the Davey-Stewartson system with cubic nonlinearity posed on the two dimensional torus. A natural setting for studying blow up solutions for this equation takes place in $H^s, 1/2 <s <1$. In this paper, we prove a lower bound on the blow up rate for these regularities.
\end{abstract}

\maketitle

\section{Introduction}

We consider the Davey-Stewartson system defined on the two dimensional torus $T^2:= \R ^2  / 2 \pi \Z ^2 $:
\begin{equation} \label{ds1}
\left \{
\begin{array}{rcl}
i \partial_t u - \partial_x^2 u + \partial_y^2 u &=&  - |u|^2u +  2u \partial_x \phi, \\
  ( \partial_x^2 + \partial_y^2 ) \phi &= & \partial_x |u|^2,
\end{array}
\right.
\end{equation}
where $u: \R \times T^2 \to \C$ and $\phi : \R \times T^2 \to \R$ are the unknowns. Rearranging the second equation, we may see this system as a dispersive equation with an hyperbolic linear part and a nonlocal nonlinearity:
\begin{equation} \label{ds2}
 i \partial_t u + P u =-|u|^2 u - E(|u|^2) u, \qquad (t,x) \in \R \times T^2,
\end{equation}
where $P=- \partial_x^2  + \partial_y ^2$ and $E$ is the nonlocal operator such that:
\begin{eqnarray*}
 \widehat{E(f)}(m,n) &=& \frac{2 m^2}{m^2+n^2} \widehat{f}(m,n), \qquad (m,n) \in \Z^2 \setminus \{ (0,0 )\}, \\
   \widehat{E(f)}(0,0) & =& 0.
\end{eqnarray*}
The Cauchy problem and the blow up theory for this equation have been studied essentially in the case where the system is posed in $\R^2$. In this case, the system is locally well posed in the Sobolev spaces $L^2, H^1$ (see \cite{GhiSau1990}) and more easily for higher regularities $H^s,  s>1$. In \cite{Oza1992}, T. Ozawa proved that the equation posed on $\R^2$ enjoys a pseudo-conformal type symmetry and as for NLS, this allows to construct a blow-up solution by applying this transformation to an explicit stationary (periodic in time for NLS) solution:
\[
 u(t,x,y)=\frac{1}{1+ x^2+y^2}.
\]
This function is then transformed into:
\begin{equation} \label{solutionexplo}
 v(t,x,y)  =  \frac{1}{a+bt} \mathrm{exp}\left (\frac{i b}{4 (a+bt)} (-x^2 +y^2) \right ) \frac{1}{ 1+  \left( \frac{x}{a+bt}\right )^2+\left ( \frac{y}{a+bt} \right)^2},
\end{equation}
with $(a,b) \in \R^2$. Note that $v(t)$ is in $H^s(\R^2)$ (see \cite{Oza1992}) for every $s<1$ with  $\|v(t)\|_{L^2} =  \sqrt{\pi}$ but is not in $H^1(\R^2)$. The solution $v$ blows up at time $T=-a/b$ in $L^2$ in the sense that the $L^2$ blow-up criteria (the $L^4_{[0,t]}L^4$ norm goes to infinity when $t$ goes to $T$) is satisfied:
\[
\| v \|_{L^4([0,t]) L^4(\R^2)} \sim \frac{C}{(T-t)^{1/4}} \to \infty \textrm{ as } t \textrm{ goes to } T,
\]
and accumulates all the mass in the origin:
\[
 |v(t)|^2 \to \pi \delta_{(0,0)} \quad \textrm{ as } t \to T \textrm{ in } \mathcal D'(\R^2).
\]
The explosion also occurs in $H^s, s <1$ with the pseudo-conformal bound (\cite{Oza1992}):
\[
 \| v(t) \|_{H^s} \sim \frac{C}{(T-t)^s}.
\]
Note that in $\R^2$, we have a scaling symmetry; if $u$ solves (\ref{ds2}) then for all $\lambda >0, (t,x,y) \mapsto \lambda u(t, \lambda ^2 t, \lambda x, \lambda y)$ also solves (\ref{ds2}). It is a classical fact that this symmetry automatically implies a lower bound on the blow up rate: for all blow up solution $u$ with maximal time $T(u)< \infty$, we have:
\begin{equation} \label{lower}
 \| u(t)\|_{H^s} \geq \frac{C}{ (T(u)-t)^{s/2}}.
\end{equation}

\medskip

By analogy with NLS in $\R^2$, we may ask the question of existence of ground states of the type $u(t,x,y)= \mathrm {exp}(i \omega t) Q(x,y)$ for (\ref{ds1}) posed on $\R^2$ with an exponentially decaying profile $Q$ but the hyperbolicity of the operator $- \partial_x^2+ \partial_y ^2$ forbids the existence of such solutions at least in the case where the nonlinearity is $-|u|^2u$ \cite{GhiSau1996}. Moreover, numerics \cite{KleMuiRoi2011} seems to show that the $L^2$-norm of the solution $u$ (or $v$) is the minimal mass for which we may have singularities in finite time. Thus, the function $u$ plays the role of a ground state but is only polynomially decaying and this requires to work with low regularities $H^s$, $s<1$.

\medskip

 The aim of this paper is twofold: first give an $H^s$ framework for studying blow up theory for (\ref{ds2}) i.e. show the local well-posedness of (\ref{ds2}) for initial data in $H^s(T^2)$ $s<1$ and secondly show that the lower bound on the blow up rate (\ref{lower}) still holds even if a scaling symmetry does not strictly make sense on the torus. The proof relies on local existence arguments on the dilated torus $\R^2 / 2 \pi L \Z^2$, $L \to \infty$ and more precisely on bilinear Strichartz estimates. The classical method \cite{BurGerTzv2005} giving well posedness from bilinear Strichartz estimates does not work in our setting because of the non local term; we will have to refine the bilinear approach with more general localizations. An interesting question, not solved here, is the localization of the solution $v$ (\ref{solutionexplo}) i.e. construct from $v$ a solution of (\ref{ds2}). The non exponential decay of $v$ is reflected in the estimate
\[
 \| v(t) \|_{H^s(  \var \leq | (x,y)| \leq A )} \leq  C (T-t)^{(1-s)},
\]
which make perturbation arguments around $v$ difficult to apply even on a compact domain. In particular, for $s>1$, blow up is not localized and this explains our choice to treat low regularities.

\begin{rem}
The system we study is called the hyperbolic-elliptic version of the generalized Davey-Stewartson system:
\begin{equation} 
\left \{
\begin{array}{rcl}
i \partial_t u + \var_1 \partial_x^2 u + \partial_y^2 u &=&  - |u|^2u +  2u \partial_x \phi, \\
  ( \var_2 \partial_x^2 + \partial_y^2 ) \phi &= & \partial_x |u|^2,
\end{array}
\right.
\end{equation}
where $\var_i \in \{-1,+1\}$. Depending on the values of $\var_i$, the local well-posedness holds \cite{GhiSau1990}, \cite{Bar2011}, \cite{LinPon1993}, \cite{Hay1997}, \cite{Chi1999} but blow up theory is really well understood only in the elliptic-elliptic case $\var_1=\var_2=1$ \cite{LiGuoJia2000}, \cite{ZhaZhu2011} \cite{Ric2011} where results are similar to those for NLS. 
\end{rem}

\section{Statement of the result and remarks}

Let us now give our result.

\begin{thm} \label{newthm} 
Let $s >1/2$. 

\smallskip

\noindent 1) The equation (\ref{ds2}) is locally well-posed in the space $H^s(T^2)$ in the following sense. There exists $b>1/2$ such that the following holds true: for all $u_0 \in H^s(T^2)$, there exist a time $T>0$ and a unique $u \in X^{s,b}_T(T^2) \subset \mathcal C([0,T], H^s(T^2))$ satisfying $u(0)=u_0$ and (\ref{ds2}). Here, $X^{s,b}_T$ denotes the Bourgain space associated to (\ref{ds2}) and defined in (\ref{defibourgain}).

\medskip

\noindent 2) Let $u_0 \in H^s(T^2)$ and $u$ the corresponding solution. If $T(u_0)$ denotes the maximal time of existence of $u$, then we have the following possibilities: either $T(u_0) =+ \infty$ or $T(u_0)< + \infty$ and in this case, there exists $C>0$ such that for all $t \in [0, T(u_0))$:
\begin{equation} \label{lower_bound}
 \|u(t) \|_{H^s(T^2)} \geq \frac{C}{(T(u_0)-t)^{s/2}}.
\end{equation} 
\end{thm}

Before giving the proof of Theorem~\ref{newthm}, let us give some comments. Consider the equation
\begin{equation} \label{nlshyper}
 i \partial_t u + Pu= -|u|^2 u, \qquad t>0, \quad (x,y) \in T^2.
\end{equation}
Strichartz type estimates hold for the operator $P$ (see \cite{GodTzv2012}, \cite{Wan2012}) and this with an analysis similar to \cite{BurGerTzv2005} gives the local well-posedness of (\ref{nlshyper}) in $H^s(T^2)$ for all $s>1/2$. Following \cite{GodTzv2012}, it is easy to check that the function defined by
\begin{equation} \label{dila}
 u(t,x,y)= e^{it | u_0 (x+y)|^2 } u_0(x+y)
\end{equation}
is a solution of (\ref{nlshyper}) for all $u_0 \in \mathcal C^{\infty} (\R / 2 \pi \Z, \R)$. If $u_0$ is not a constant function and if $s \geq 0$, we can check that there exists $C>0$ such that for $t \geq 1$,
\begin{equation} 
 \| u(t)\|_{H^s} \geq Ct^s.
\end{equation}
Thus, for $s>1/2$, we obtain an explicit solution which blows up in infinite time; this contrasts with the usual Schr\"odinger equation. Using a suitable rescaling of the explicit solution (\ref{dila}), one may also show the local ill-posedness of (\ref{nlshyper}) in $H^s, s<1/2$ (see the appendix of \cite{BurGerTzv2005a} for a similar discussion). It would be interesting to know if we may construct solutions behaving like (\ref{dila}) for the equation (\ref{ds2}).

\section{Proof of the result}

\textbf{Strategy of the proof.} To prove Theorem~\ref{newthm}, the idea is to rescale the torus $T^2=\R^2 / 2 \pi \Z^2$ by considering $T^2_L=\R^2 /  2 \pi L \Z^2$ where $L>0$ will tend to infinity. In a first step, we perform a Banach fixed point argument in the dilated Bourgain space $X^{s,b}_{T, L}$ to obtain a local well-posedness result in $H^s(T_L^2)$ with the bound on the blow up time:
\begin{equation} \label{time}
 T(u_0) \geq F( \|u_0\|_{H^s(T_L^2)}),
\end{equation}
for some function $F$ independent of $L$. This relies on a uniform bilinear Strichartz estimate. In our analysis, it is of importance that dispersive estimates are local in space and time. In particular, if we take $L=1$, this step will give the first point of Theorem~\ref{newthm}. In the step 2, we deduce the blow up lower bound from a scaling argument. The bound (\ref{lower_bound}) which is the same as $\R ^2$ is in accordance with the fact that when $L$ goes to infinity $T^2_L$ looks like $\R^2$ formally. Note that the machinery of the Bourgain spaces is natural to study such questions but we do not exclude the possibility of working on other spaces by adapting harmonic analysis results to the case of $T_L^2$ to treat the operator $E$.

\medskip

\textbf{Notations.} We denote by $e_{m,n}(x,y) = (2 \pi)^{-1} \mathrm{exp}( i mx + i ny)$ the usual orthonormal basis of $L^2(T^2_1)$. When working on $T^2_L$, we will keep the same notation $e_{m,n}$ for the rescaled basis: $e_{m,n}(x,y) = (2 \pi L)^{-1} \mathrm{exp}( i (m/L)x + i (n/L)y)$. For a function $u$ defined on $T_L^2$, we note
\[
\Delta_Q (u) = \sum_{(m,n) \in Q} c(m,n) e_{m,n},
\]
where $c(m,n)$ are the Fourier coefficients of $u$:
\[
c(m,n) = \frac{1}{2 \pi L} \int_{T^2_L} u(x,y) e^{ i( \frac{m}{L}x + \frac{n}{L}y)} dx dy.
\]
If $Q  \subset \Z^2$ and $R$ is a dyadic number, we set for a function $u(t,x)$ defined on $\R \times T_L^2$: 
\be
\Delta_{Q,R} u =  \sum_{(m,n) \in Q} \left (  \int_{ R \leq \langle \tau - \frac{m^2}{L^2}+\frac{n^2}{L^2}  \rangle \leq 2 R } \widehat{c_{n,m}}(\tau) e^{2 i \pi  t \tau}  d \tau   \right ) e_{m,n},
\ee
where $c_{m,n}(t)$ are the Fourier coefficients of $u(t)$. When $Q$ is the cube $Q=\{ (m,n) \in \Z^2, N \leq   \mathrm{Max}( |m/L|, |n/L| ) \leq 2N \}$, we will note $\Delta_N = \Delta_Q$ and $\Delta_{N,R} = \Delta_{Q,R}$.

\medskip

\textbf{Step 1.} We prove: for all $s >1/2, L \geq 1$ and $u_0 \in H^s(T_L^2)$, there exists a solution $u$ of 
\begin{equation} \label{eq2}
 i \partial_t u + Pu= -|u|^2 u -E(|u|^2) u, \quad  (x,y) \in T^2_L,
\end{equation}
and $\alpha >0$, $C>0$ independent of $L$ satisfying the lower bound on the blow up time:
\begin{equation} \label{rate1}
T(u_0) \geq \frac{C}{ \|u_0\|_{H^s(T_L^2)}^{\alpha}}.
\end{equation}
The point here is that the lower bound depends only on the size of the initial data and not on $L$. On $T_L^2$, we denote (without changing notations) by $P$ and $E$ the natural extensions of the operators $P$ and $E$ defined above on $T^2$. Hence, symbols are respectively $(-m^2+n^2)/L^2$ and $2 m^2/(m^2+n^2)$. 

\medskip

\textbf{High regularity.} Before looking at low regularities and to convince the reader that \ref{rate1}) holds, let us focus on the easier case of more regular data i.e. $s \in \N \setminus \{0,1 \}$. Let us prove (\ref{rate1}) in this case. Let $L \geq 1$ and consider the equation (\ref{eq2}) and its equivalent formulation 
\begin{equation} \label{intformulation}
u(t) = e^{it P} u_0 + i \int_0^t e^{ i(t-\tau) P  }\left ( |u(\tau)|^2 u(\tau) +E( | u(\tau)|^2) u(\tau) \right ) d \tau.
\end{equation}
Let $u_0 \in H^s(T^2_L)$. Taking the $H^s$-norm in (\ref{intformulation}) and using the triangle inequality, we get for a constant $C>0$ independent of $L$:
\beq \label{esti1}
\|u(t)\|_{H^s} \leq \|u_0 \|_{H^s} + C \int_0^T \left ( \| u(\tau) ^2 u (\tau) \|_{ H^s} + \|E( | u(\tau)|^2) u(\tau) \|_{H^s} \right ) d \tau.
\eeq
Now we need a Sobolev type inequality with constants independent of the size of the torus. 
\begin{lem} \label{sobolevembedding}
Let $s$ be an integer with $s \geq 2$. There exists a constant $C>0$ such that for all $L>0$ and $v, w \in H^s(T_L)$, 
\[
\|vw\|_{H^s(T^2_L)} \leq C \| v\|_{H^s(T^2_L)} \|w\|_{H^s(T^2_L)}.
\]
\end{lem}

\begin{proof}
We first prove that for all $\sigma>1$ and $v \in H^{\sigma} (T_L), \|v\|_{L^{\infty} } \leq C \|v\|_{H^{\sigma}}$ for a constant $C>0$ depending only on $\sigma$. Indeed, expanding $v$ in Fourier series, we first get
\[
| v| \leq \frac{1}{2 \pi L} \sum_{(m,n) \in \Z^2} |v_{m,n}|.
\]
We make appear the $H^{\sigma}$-norm of $u$ and use Cauchy-Schwarz inequality to obtain:
\begin{eqnarray}
\|v\|_{L^{\infty}} & \leq& \frac{1}{2 \pi L} \sum_{(m,n) \in \Z^2} |v_{m,n}| \left ( 1 + \frac{m^2}{L^2} + \frac{n^2}{L^2} \right )^{\sigma/2}  \left ( 1 + \frac{m^2}{L^2} + \frac{n^2}{L^2} \right )^{-\sigma /2} \nonumber \\
   & \leq & \frac{1}{2 \pi L} \| v\|_{H^{\sigma}} \left ( \sum_{ (m,n) \in \Z^2} \left ( 1+\frac{m^2}{L^2} + \frac{n^2}{L^2} \right )^{-\sigma} \right )^{1/2}. \label{esti2}
\end{eqnarray}
But we can easily compute the dependence in $L$ of the last sum above by comparing with an integral as follow: 
\begin{eqnarray*}
\sum_{ (m,n) \in \Z^2} \left ( 1+\frac{m^2}{L^2} + \frac{n^2}{L^2} \right )^{-\sigma} & \leq & C \sum_{ (m,n) \in \Z^2} \left ( 1+\frac{m^2}{L^2}  \right )^{-\sigma/2}  \left ( 1+\frac{n^2}{L^2}  \right )^{-\sigma/2} \\
 & \leq & C  \left ( \int \frac{dx}{ \left ( 1+ \frac{x^2}{L^2} \right ) ^{\sigma/2} } \right )^2 \\
 & \leq & C L^2.
\end{eqnarray*}
Thus, there is not more dependence on $L$ in (\ref{esti2}) and we obtain the claim. Now we can prove the lemma. Indeed, we first write the Leibniz rule then use the previous claim and an interpolation argument to get:
\begin{eqnarray*}
\| (- \Delta )^{s/2}  (vw)\| _{L^2}  & \leq & C  \| v\|_{H^s} \| w \|_{H^s} .
\end{eqnarray*}
Here, the constant $C$ contains binomial coefficients and therefore is independent of $L$. Moreover, again with the embedding $H^s \hookrightarrow L^{\infty}$, we have
\[
\|vw\|_{L^2} \leq \|v \|_{L^{\infty}} \|w \|_{L^2} \leq C \| v\|_{H^s} \|w\|_{H^s},
\]
and the last two inequalities end the proof of the lemma.
\end{proof}
Therefore, coming back to (\ref{esti1}), using the boundedness of $E$ in $H^s$ and Lemma~\ref{sobolevembedding}, we have
\[
\| u(t) \|_{H^s} \leq \|u_0\|_{H^s} +C T \| u(t)\|_{H^s}^3,
\]
with $C>0$ independent of the period. This last estimate allows us to perform a Banach fixed point argument (the Lipschitz property is proved with similar arguments) in a ball of the space $\mathcal C( [0,T], H^s)$ of radius $M=2 \|u_0\|_{H^s}$ and with $T =C/ \| u_0\|_{H^s}^2$, and this proves (\ref{rate1}).

\textbf{Low regularity.} This part is the more interesting since, as said above, the explicit blow up solution $v(t)$ of the introduction lives only in $H^s$ with $s<1$. So let $s >1/2$. We define the Bourgain spaces associated with equation (\ref{eq2}) as the completion of the space of smooth compactly supported functions on $\R \times T_L^2$ for the norm defined by:
\be
\| u\|_{X^{s,b}_L} = \left \| \langle i \partial_t + P \rangle  ^{b} \langle (- \lap )^{ \frac{1}{2}} \rangle ^{s} u \right \|_{L^{2}( \R \times T^2_L)},
\ee
where $\langle \alpha \rangle = (1+ \alpha^2)^{1/2}$. Note that there exist more convenient equivalent definitions of this space: we may also check that the norm is equivalent to the following
\be
\| u\|_{X^{s,b}_L} ^2 = \sum_{(m,n) \in \Z ^2} \left (1+\frac{m^2}{L^2}+\frac{n^2}{L^2} \right )^{s} \int_{\R} \langle \tau - \frac{m^2}{L^2} + \frac{n^2}{L^2} \rangle ^{2b} | \widehat{c_{m,n}} (\tau) |^2 d \tau,
\ee
where $\widehat{c_{m,n}}$ is the Fourier transform of $c_{m,n}$. A last definition is possible linking the Bourgain norm with Sobolev norm of the free dynamic:
\begin{equation} \label{definition}
\| u\|_{X^{s,b}_L} = \| e^{-i t P} u(t) \|_{H^b (\R, H^s(T^2_L))}.
\end{equation}
We will work on a finite time interval so that we have to define the localized version of the Bourgain spaces; for $u : [0,T] \times T_L^2 \to \C$:
\begin{equation}  \label{defibourgain}
\| u \|_{X^{s,b}_{L,T}} = \inf \left \{ \| v\|_{X^{s,b}_L}, \  v \in X^{s,b}_L \textrm{ such that } v(t)=u(t) \textrm{ for all } t \in [0,T] \right  \}.
\end{equation}
Let us recall the integral formulation (\ref{intformulation}):
\be
u(t) = e^{it P} u_0 + i \int_0^t e^{ i(t-\tau) P  }\left ( |u(\tau)|^2 u(\tau) + E( | u(\tau)|^2) u(\tau) \right ) d \tau.
\ee
First, the linear term is easily bounded: if $T \leq 1$ and $\psi (t)$ denotes a smooth real cut-off function equal to $1$ on $[0,1]$ and with compact support, we get using the definition of the Bourgain spaces (\ref{definition}):
\begin{equation} \label{linearestimate}
\| e^{i t P} u_0 \|_{X^{s,b}_{L,T}} \leq \| e^{it P} \psi (t) u_0 \|_{X^{s,b}_L} \leq \| \psi (t) u_0 \|_{H^b (\R, H^s(T_L^2))} \leq  C \| u_0 \|_{H^s (T^2_L)} ,
 \end{equation}
where $C=\| \psi \|_{H^b(\R)} $ is independent of the period $L$.
\begin{lem} \label{lemmarien}
 There exists $C>0$ such that for all $L \geq 1$, $T \leq 1$ and all pair $(b,b')$ satisfying $0 < b'<1/2 <b,\  b+b' <1$, 
\be
\left \| \int_0^t e^{i (t-\tau) P} F(\tau) d \tau \right \|_{X^{s,b}_{L,T}} \leq C T^{1-b-b'} \| F\|_{X^{s,-b'}_{L,T}}.
\ee
\end{lem}

\begin{proof}
For a fixed $L>0$, this estimate is classical in the context of the Bourgain spaces. To see that we may choose $C$ independent of $L$, we remark (see \cite{BurGerTzv2005}) that the proof of such an estimate for a fixed $L$ relies on the one dimensional inequality (proved in \cite{Gin1995}):
\begin{equation} \label{pointwise}
\| \phi (\frac{t}{T}) \int_0^t g( \tau) d \tau \| _{H^b(\R)} \leq C T^{1-b-b'} \| g\|_{H^{-b'} (\R)},
\end{equation}
for a cut-off function $\phi$.
Then we apply this estimate pointwise with $g(\tau)=(F(\tau, x), e_{m,n} ) e_{m,n}$, take the square, integrate on $T^2_L$, multiply by $(-m^2+n^2)/L^2$ and sum for $(m,n) \in \Z^2$. We then obtain the desired estimate with the same constant $C$ as in (\ref{pointwise}) thus independent of $L$.
\end{proof}

To treat the nonlinearity in the fixed point argument, we will need the following proposition.
\begin{pro}[\textbf{Trilinear estimate}] \label{trilinear}
There exist a pair $(b,b')$ satisfying $0 < b' < 1/2 < b, b+b' <1$ and a constant $C>0$ such that for every $L \geq 1, T>0$, $u_1, u_2,u_3 \in X^{s,b}_{L,T}$,
\[
 \| u_1 u_2 u_3 \|_{X^{s,-b'}_{L,T}} \leq C \| u_1\|_{X^{s,b}_{L,T}} \| u_2  \|_{X^{s,b}_{L,T}} \| u_3\|_{X^{s,b}_{L,T}},
\]
\[
 \| E(u_1 u_2 )    u_3 \|_{X^{s,-b'}_{L,T}}  \leq C \| u_1\|_{X^{s,b}_{L,T}} \| u_2  \|_{X^{s,b}_{L,T}} \| u_3\|_{X^{s,b}_{L,T}}.
\]
\end{pro}

\begin{proof}
 Let us start with a lemma.

\begin{lem}[\textbf{Uniform periodic bilinear Strichartz estimate}] \label{strichartzuniforme}
There exists $C>0$ such that for every $N_1, N_2 \geq 1$ dyadic numbers, $(a_1,b_1), (a_2,b_2) \in \Z^2$, $L \geq 1$ and $u_1, u_2 \in L^2(T^2_L)$ writing
\be
u_1= \sum_{N_1 \leq \mathrm{Max}\left (\left |\frac{m}{L} -a_1\right |, \left |\frac{n}{L} -b_1\right | \right ) \leq 2 N_1} c_1(m,n) e_{m,n}, \qquad u_2= \sum_{N_2 \leq \mathrm{Max}\left (\left |\frac{m}{L} -a_2\right |, \left |\frac{n}{L} -b_2\right | \right ) \leq 2N_2} c_2(m,n) e_{m,n}
\ee 
we have the bilinear estimate 
\begin{equation} \label{bil}
\| e^{it P} (u_1) e^{it P} (u_2) \|_{L^2([0,1 ]) L^2(T_L^2)} \leq C \mathrm{min}\left (N_1, N_2 \right ) ^{1/2}  \| u_1 \|_{L^2(T_L^2)} \| u_2 \|_{L^2(T_L^2)}.
\end{equation}
\end{lem}

\begin{proof}[Proof of Lemma~\ref{strichartzuniforme}] Note that for $L=1$, linear Strichartz estimates have been proved recently in \cite{Wan2012}, \cite{GodTzv2012}. We first prove the property in the case where $u_1=u_2$ and $a_1=b_1=a_2=b_2=0$. So let $u=u_1=u_2$ and $N=N_1=N_2$. We recall the semiclassical Strichartz estimate on the torus of size $1$ (see \cite{GodTzv2012}): for all $h \in (0,1)$, for all interval $J$ of size $h$ and for all $u_0$ writing
\[
v_0 = \sum_{h^{-1} \leq  \mathrm{Max}( |m|, |n| ) \leq 2 h^{-1}} c(m,n) e_{m,n},
\]
for some coefficient $c(m,n)$, we have 
\begin{equation} \label{str}
\| e^{it P } v_0 \|_{L^4 (J) L^4 (T_1^2)} \leq C \|v_0\|_{L^2(T_1^2)} .
\end{equation}
Similarly to the case where $P$ is the Laplace operator (see \cite{Han2012}), we apply a scaling argument on this estimate to derive a linear Strichartz estimate on $T_L^2$ on the time interval $[0,1]$. Let $u_0 \in L^2(T_L^2)$ localized in frequency in $[0,N]$ i.e.
\begin{equation}
u_0= \sum_{ N \leq \mathrm{Max} \left (\left | \frac{m}{L} \right |, \left | \frac{n}{L} \right | \right )\leq 2N } c(m,n) e_{m,n},
\end{equation}
and $v_0 \in L^2(T_1^2)$ defined by $v_0(x)=u_0(Lx)$. Then computing the $L^4 ([0,1]) L^4 (T_L^2)$ of $\mathrm{exp}(it P) u_0$ in term of $v_0$ and applying a change of variable, we get
\[
\| e^{it P} u_0 \|_{L^4([0,1]) L^4(T_L^2)} = L \| e^{it P} v_0 \|_{ L^4 ([0, L^{-2}]) L^4(T_1^2)}. 
\] 
Remark that $v_0$ writes
\[
v_0 = \sum_{LN \leq \mathrm{Max} \left (|m|, |n|\right  ) \leq 2LN  } c(m,n) e_{m,n},
\]
so that we may apply (\ref{str}) with $h \sim LN$. We need to consider two cases. If $L \geq N$, then $[0, L^{-2}] \subset [0, (LN)^{-1}]$ and so
\[
\| e^{it P} u_0 \|_{L^4([0,1]) L^4} \leq L \| e^{it P} v_0 \|_{L^4([0,(LN)^{-1} ]) L^4}  \leq C L \|v_0\|_{L^2} \leq C \|u_0\|_{L^2} .
\]
If $L <N$, we write $[0, L^{-2}]$ as a union of intervals $[t_k, t_{k+1}]$ with $t_{k+1}- t_k \sim (LN)^{-1}$ and $k \sim N/L$. We apply (\ref{str}) on each $[t_k, t_{k+1}]$ and this gives:
\[
\| e^{it P} u_0 \|_{L^4([0,1]) L^4} \leq C \left ( \frac{N}{L} \right )^{1/4} L \|v_0\|_{L^2} \leq  C \left ( \frac{N}{L} \right )^{1/4} \|u_0\|_{L^2}.
\]
If $L \geq 1$, we may in particular summarize the last two inequalities as 
\[
\| e^{it P} u_0 \|_{L^4([0,1]) L^4} \leq C N^{1/4} \|u_0 \|_{L^2},
\] 
and this proves (\ref{bil}) if $u_1=u_2$ and $a_1=b_1=a_2=b_2=0$. Now we treat the case $u=u_1=u_2$ but without the assumption $a_1=b_1=a_2=b_2=0$. Let $(a,b) \in \Z^2$ and write
\begin{eqnarray*}
u&=&\sum_{ N \leq \mathrm{Max} \left ( \left | \frac{m}{L}-a \right |,   \left | \frac{n}{L}-b \right | \right ) \leq 2N} e^{\frac{it}{L^2}(n^2-m^2)} c(m,n) e_{m,n} \\
  & =& e^{iax} e^{iby} e^{ it (a^2-b^2)} \sum_{N \leq \mathrm{Max} \left ( \left | \frac{p}{L} \right |,   \left | \frac{q}{L} \right | \right ) \leq 2N } c(aL+p,bL+q) e^{\frac{-it}{L^2} (p^2-q^2+2aLp-2bLq)} e_{p,q}.
\end{eqnarray*}
Then
\begin{eqnarray*}
\| u\|_{L^4 L^4}^4 &=& \| \sum_{N \leq \mathrm{Max} \left ( \left | \frac{p}{L} \right |,   \left | \frac{q}{L} \right | \right ) \leq 2N}  \frac{1}{2 \pi L} c(aL+p, bL+q) e^{\frac{-it}{L^2}(p^2-q^2)} e^{ \frac{ip}{L} (x-2at)}  e^{ \frac{iq}{L} (y+2bt)} \|_{L^4 L^4 }^4  \\
                   & =& \int_t \int_{\substack{  -2at  \leq \alpha \leq 2 \pi L -2at  \\ 2bt \leq \beta \leq 2 \pi L + 2bt   }} \left | \sum_{N \leq \mathrm{Max} \left ( \left | \frac{p}{L} \right |,   \left | \frac{q}{L} \right | \right ) \leq 2N} e^{\frac{-it}{L^2}(p^2-q^2)} \frac{1}{ 2 \pi L} c(aL+p,bL+q) e^{\frac{i}{L} p \alpha } e^{\frac{i}{L} q\beta } \right | ^4 d \alpha d \beta dt \\
                   & =&  \int_t \int_{\substack{   0 \leq \alpha \leq 2 \pi L  \\ 0\leq \beta \leq 2 \pi L   }} \left | \sum_{N \leq \mathrm{Max} \left ( \left | \frac{p}{L} \right |,   \left | \frac{q}{L} \right | \right ) \leq 2N} e^{\frac{it}{L^2}(q^2-p^2)} c(aL+p,bL+q) e_{p,q}(\alpha, \beta) \right | ^4 d \alpha d \beta dt .
\end{eqnarray*}
We apply the linear result proved above with $(a,b)=(0,0)$ and this gives
\begin{eqnarray*}
\| u\|_{L^4 L^4}^4 & \leq &CN\left (\sum_{k,l} | c(k,  l) |^2  \right ) ^2 \\
 & \leq  & C N  \|u_0\|_{L^2}^4.
\end{eqnarray*}
This proves the result when $u_1=u_2$. Note that if we assume another type of localization for $u$
\[
 u= \sum_{ \mathrm{Max}( | \frac{m}{L} -a|, | \frac{n}{L}- b | ) \leq N } c(m,n) e_{m,n},
\]
the $L^4L^4$ estimate still holds. It may be seen by remarking that estimate (\ref{str}) also holds if $u_0$ is spectrally localized in $\{ (m,n) \in \Z^2,   \mathrm{Max}(|m|, |n| ) \leq 2 h^{-1} \}$ (see \cite{GodTzv2012}) and using the same analysis as above. Now, we can prove the bilinear estimate in the general case. We assume for instance $N_1 \leq N_2$ and decompose the set $A= \{ (m,n) \in \Z^2, N_2 \leq \mathrm{Max} (| a_2-m/L | ,| b_2-n/L|) \leq 2N_2 \} $ in small disjoint cubes of the form $ Q_{\alpha}=Q_{(k,l)} = \{ (m,n) \in A,\mathrm{Max} (| k-m/L | ,| l -n/L|) \leq N_1  \} $ for $\alpha =(k,l)$ running over a set $I$. Then for different $\alpha's$, the functions $e^{it P}(u_0) e^{it P} (\Delta_{Q_{\alpha}}  v_0)$ are almost orthogonal since each function is localized in Fourier in the set $D_{\alpha}:= \{ (m,n) \in \Z^2, N_1 \leq \mathrm{Max}( |m/L -a_1|, |n/L -b_1 | ) \leq 2N_1 \} + Q_{\alpha}$ and the sets $D_{\alpha}$ are almost disjoint in the sense that each point of $\Z^2$ belongs to a 
finite number of sets $D_{\alpha}$. Indeed, if $(m,n) \in D_{\alpha_1} \cap D_{\alpha_2}$, then in particular we may write
\[
 (m,n)=c+d=e+f,
\]
with $d \in Q_{\alpha_1}$, $f \in Q_{\alpha_2}$, and $c, e \in \{ (m,n) \in \Z^2, N_1 \leq \mathrm{Max}( |m/L -a_1|, |n/L -b_1 | ) \leq 2N_1 \} $. We deduce $|c_1-e_1| =|f_1-d_1| \leq 4 N_1L$. But each $Q_{\alpha}$ is of size less than $4 N_1L$ and there is a finite number of $Q_{\alpha}$ whose distance to $Q_{\alpha_2}$ is less than $4 N_1 L$. So if we fix $\alpha_1$, then $\alpha _2$ runs in a finite number of indexes. Thus, this orthogonality property implies 
\begin{eqnarray*}
 \| e^{it P}(u_1) e^{it P}(u_2) \|_{ L^2 L^2}^2  & \leq & C \sum_{ \alpha \in I} \| e^{it P}(u_1) e^{it P}(\Delta_{Q_{\alpha}} u_2) \|_{ L^2 L^2}^2 \\
     & \leq & C \| e^{it P}(u_1) \|_{L^4 L^4 }^2  \sum_{ \alpha \in I}  \| e^{it P}(\Delta_{Q_{\alpha}}  u_2) \|_{L^4 L^4 }^2 \\
     & \leq & C N_1  ^{1/2} \| u_1 \|_{L^2} ^2 N_1  ^{1/2}  \sum_{ \alpha \in I} \| \Delta_{Q_{\alpha}} u_2 \|_{L^2}^2 \\
     & \leq & C N_1  \| u_1 \|_{L^2} ^2 \| u_2 \|_{L^2}^2.
\end{eqnarray*}
This proves the proposition.
\end{proof}

\begin{rem}
Note that if $Q_i$ denotes the set 
\[
Q_i=\left \{ (m,n) \in \Z^2, N_i \leq \mathrm{Max} \left ( \left | \frac{m}{L} -a_i \right |,  \left | \frac{n}{L} -b_i \right | \right )  \leq 2N_i \right \},
\]
then $(N_iL)^2 \leq |Q_i|  $ and we may rewrite the Strichartz estimate (\ref{bil}) as
\begin{equation} \label{cube0}
\| \Delta_{Q_1} (e^{it P}u ) \Delta_{Q_2} (e^{it P}v ) \|_{L^2 L^2} \leq C \left ( \frac{ \mathrm{min} (|Q_1|, |Q_2 |)}{L^2} \right ) ^{\frac{1}{4}}\| u\|_{L^2} \| v\|_{L^2} .
\end{equation}
Once we have proved (\ref{cube0}), from covering arguments, we may deduce the same estimate for other shapes of $Q_i$ typically $Q_i= \{ (m,n) \in \Z^2, \mathrm{Max}( |m/L| , |n/L|) \leq 2N_i  \}$ or translated sets of the previous one. In the sequel, we will use (\ref{cube0}) for these kinds of $Q_i$. More precisely, we have the following.
\end{rem}

\begin{lem} \label{cube}
 For all $b>1/2$, there exist $C(b)>0$, $\beta(b) \in (0, 1-b)$ and $\var(b)>0$ such that for all dyadic square $Q_1, Q_2 \subset \Z ^2$, $R_1, R_2$ dyadic number, $L \geq 1$ and $u_0, v_0 \in L^2( \R, L^2(T^2_L))$,
\begin{multline} \label{proj}
\| \Delta_{Q_1,R_1} u_0 \   \Delta_{Q_2,R_2} v_0 \|_{L^2 L^2} \leq C(b)  \left ( \frac{\mathrm{Min}(|Q_1|, |Q_2|)}{L^2} \right )^{1/4+ \var(b)} (R_1 R_2)^{\beta(b)}  \\ \times \| \Delta_{Q_1,R_1} u_0 \|_{L^2 L^2} \| \Delta_{Q_2,R_2} v_0 \|_{L^2 L^2} ,
\end{multline}
where $|Q_i|$ denotes the number of points in $Q_i$. Moreover, we may choose $\var (b)$ such that $\var(b)$ goes to $0$ as $b$ goes to $1/2$.
\end{lem}

\begin{proof}[Proof of Lemma~\ref{cube}]
As for the proof of (\ref{bil}), we first assume $u=u_0=v_0$. Next, from (\ref{bil}), we get for all $b >1/2$ and $f \in X^{0,b}_L$ localized in frequency in $Q$,
\[
\| f\|_{L^4 L^4} \leq C \left ( \frac{|Q|}{L^2} \right )^{1/8} \| f\|_{X^{0,b}_L}.
\] 
Again the constant $C$ does not depend on $L$ since the proof (see \cite{BurGerTzv2005}) relies only on manipulations in time. In particular, for all $u$,
\beq \label{interpolation}
\| \Delta_{Q,R}  u \|_{L^4 L^4} \leq C \left ( \frac{|Q|}{L^2} \right )^{1/8} \| \Delta _{Q,R}  u  \|_{X^{0,b}_L} ,
\eeq
for all $b > 1/2$. And this gives using properties of the Bourgain spaces
\begin{eqnarray} 
  \| \Delta_{Q,R}  u \|_{L^4 L^4} \leq  C \left ( \frac{|Q|}{L^2} \right )^{1/8} R^{b} \| \Delta _{Q,R} u \|_{L^2 L^2}. \label{estimate}
\end{eqnarray}
The fact that $b > 1/2$ in the above estimate will not be enough to conclude so that we need to refine this $L^4L^4$ estimate. To do so, we compute the $L^{\infty} L^{\infty}$ norm of $\Delta_{Q,R} u $. From the definition of the projection $\Delta _{Q,R}$, we get using twice Cauchy-Schwarz inequality
\begin{eqnarray}
\| \Delta_{Q,R}u \|_{L^{\infty} L^{\infty}} & \leq & \frac{1}{L} \sum_{(m,n) \in Q} \int_{ R \leq \langle \tau - \frac{m^2}{L^2} + \frac{n^2}{L^2} \rangle \leq 2R } |\widehat{c_{m,n}}(\tau)|   d \tau  \nonumber \\
 & \leq & \frac{R^{1/2}}{L} \sum_{(m,n) \in Q} \left ( \int_{ R \leq \langle \tau - \frac{m^2}{L^2} + \frac{n^2}{L^2} \rangle \leq 2R} | \widehat{c_{m,n}}(\tau) | ^2 d \tau \right )^{1/2} \nonumber \\
  & \leq & R^{1/2} \left ( \frac{ |Q|}{L^2} \right )^{1/2} \left ( \sum_{(m,n) \in Q} \int_{ R \leq \langle \tau - \frac{m^2}{L^2} + \frac{n^2}{L^2} \rangle \leq 2R} | \widehat{c_{m,n}}(\tau) |^2 d \tau \right )^{1/2} \nonumber \\
  & \leq & \left ( \frac{ |Q|}{L^2} \right )^{1/2} R^{1/2} \| \Delta_{Q,R} u \|_{L^2 L^2} . \label{etoile}
\end{eqnarray}
 By interpolation between the trivial inequality $ \| \Delta_{Q,R} u\|_{L^2 L^2} \leq \| \Delta_{Q,R} u \|_{L^2 L^2}$ and (\ref{etoile}), we have
\begin{equation} \label{interpolg}
\| \Delta_{Q,R} u \|_{L^4 L^4} \leq \left ( \frac{|Q|}{L^2} \right )^{1/4} R^{1/4} \| \Delta _{Q,R} u \|_{L^2L^2}.
\end{equation}
Let $\var(b)>0$ such that $\delta(b) :=b(1-8 \var(b)) +8 \var(b)\frac{1}{4} \in (0, 1-b)$ and $\var(b) \to 0$ as $b \to 1/2$. For instance choose $\delta(b)=3/2-2b$. Next, by interpolation between (\ref{estimate}) with weight $1- 8 \var(b)$ and (\ref{interpolg}) with weight $8 \var(b)$, we get the expected estimate: 
\begin{equation} \label{bil12}
\| \Delta_{Q,R} u_0  \|_{L^4 L^4} \leq C \left (  \frac{|Q|}{L^2} \right )^{1/8+ \var(b)} R^{\delta (b)} \| \Delta_{Q,R}  u_0 \|_{L^2L^2} .
\end{equation}
To deduce (\ref{proj}) from (\ref{bil12}), we proceed as in the proof of Strichartz estimate (\ref{bil}): if for instance $|Q_1| < |Q_2|$ then we decompose $Q_2$ in pieces of size $|Q_1|$ and next apply an almost orthogonality argument. We omit this argument and the proof is over.
\end{proof}

To prove Lemma~\ref{trilinear}, it is enough to prove the trilinear estimate for the global space $X^{s,b}_L$ i.e. $T = \infty$, then we recover the local in time estimate by taking the infimum on all extensions of $u_1, u_2, u_3 \in X^{s,b}_{L,T}$. Moreover, we only prove the second estimate; the first one is easier. By a duality argument, we have to show the quadrilinear estimate: there exists $C>0$ such that for all $L \geq 1$, $u_1, u_2,u_3, u_4 \in X^{s,b}_L$:
\be
\left | \int_{ \R \times T^2_L} E(u_1 u_2) u_3 u_4 \right | \leq  C \| u_1\|_{X^{s,b}_L} \| u_2  \|_{X^{s,b}_L} \| u_3\|_{X^{s,b}_L} \| u_4\|_{X^{-s,b'}_L}.
\ee 
In the sequel, we will note $Q_i=\{ (m,n) \in \Z^2, N_i \leq \mathrm{Max} (|m/L| ,|n/L|) < 2 N_i \} $. Decomposing each $u_i$ as
\[
u_i= \sum_{N_i, R_i} \Delta_{N_i, R_i} (u_i),
\]
we have that 
\[
G=\int_{ \R \times T^2_L} E(u_1 u_2) u_3 u_4
\]
becomes
\[
 G=\int_{\R \times T^2_L} \sum_{\substack{ N_1,N_2,N_3,N_4 \\ R_1,R_2,R_3,R_4}} E\left(\Delta_{N_1,R_1}(u_1)\Delta_{N_2,R_2}(u_2)\right) \Delta_{N_3,R_3}(u_3) \Delta_{N_4,R_4}(u_4).
\]
In the summation above, we may restrict indexes to $N_4  \leq 2 (N_1 +N_2+N_3)$. Indeed, the function 
\[
U=E\left(\Delta_{N_1,R_1}(u_1)\Delta_{N_2,R_2}(u_2)\right) \Delta_{N_3,R_3}(u_3)
\]
is localized in Fourier in the set $\{ (m,n) \in \Z^2, m=m_1+m_2+m_3, n=n_1+n_2+n_3 , (m_1,n_1) \in Q_1, (m_2,n_2) \in Q_2, (m_3,n_3) \in Q_3 \}$. Thus, if $N_4  > 2(N_1 +N_2+N_3)$, the integral over $T_L^2$ of $U\Delta_{N_4, R_4} (u_4)$ is zero. Therefore
\begin{equation} \label{summation}
 G= \sum_{\substack{ N_4 \leq 2(N_1+N_2+N_3) \\ R_1,R_2,R_3,R_4}} \alpha (N_1,N_2,N_3,N_4, R_1, R_2, R_3, R_4),
\end{equation}
where 
\[
 \alpha (N_1,N_2,N_3,N_4, R_1, R_2, R_3, R_4) = \int_{\R \times T^2_L} E\left(\Delta_{N_1,R_1}(u_1)\Delta_{N_2,R_2}(u_2)\right) \Delta_{N_3,R_3}(u_3) \Delta_{N_4,R_4}(u_4).
\]
Contrary to the case of a typical cubic nonlinearity, $\alpha$ is not symmetric in $N_1, N_2, N_3, N_4$ and we need to split the analysis in several cases. The worst situation is when the two lowest frequencies appear in the nonlocal term. Let us first treat this case.

 \medskip

\textbf{Case} {\mathversion{bold} $N_3= \mathrm{max}(N_1,N_2,N_3).$} Without loss of generality, we may assume $N_1 \leq N_2 \leq N_3$. In this situation, we decompose the set $Q_3$ in small pieces of size $N_2 L$. Hence, we may write $Q_3$ as a disjoint union of sets of the form $Q_{\alpha}=Q_{(a,b)} = \{(m,n) \in Q_3, \mathrm{Max} ( |a-m/L|, |b-n/L | ) \leq N_2  \}$ for some well chosen set $I$ of pairs $\alpha = (a,b) \in Q_3$ so that the union is disjoint. Using again an orthogonality argument, $\alpha$ is then
\[
 \alpha(N_i, R_i) =  \int_{\R \times T^2_L}  E\left(\Delta_{N_1,R_1}(u_1)\Delta_{N_2,R_2}(u_2)\right) \Delta _{Q_{\alpha},R_3}(u_3)  \Delta_{\tilde{Q}_{\alpha}, R_4} (u_4) 
\]
where 
\[
\tilde{Q}_{\alpha}=\{ (m_4,n_4) \in Q_4, m=-m_1-m_2-m_3, n=-n_1-n_2-n_3, (m_i,n_i) \in Q_i, i=1,2, (m_3,n_3) \in Q_{\alpha}\}.
\]
From Cauchy-Schwarz inequality in space and time and the boundedness of $E$ on $L^2(T_L^2)$, 
\[
 | \alpha(N_i, R_i) | \leq \| \Delta_{N_1,R_1}(u_1)\Delta_{N_2,R_2}(u_2)\|_{L^2 L^2} \| \Delta _{Q_{\alpha},R_3}(u_3)  \Delta_{\tilde{Q}_{\alpha}, R_4} (u_4) \|_{L^2 L^2}.
\]
Note that since $| Q_{\alpha } | \leq (N_2L)^2$, we deduce by the triangle inequality that we also have $| \tilde{Q}_{\alpha} | \leq C (LN_2)^2$ and thus we can apply Lemma~\ref{cube} to get 
\begin{multline}
 \alpha(N_i, R_i)   \leq  C  N_1^{\frac{1}{2}+  \var(b)} N_2^{\frac{1}{2}+ \var(b)} ( R_1 R_2 R_3 R_4)^{\beta (b)}    \| \Delta_{N_1, R_1} (u_1) \| _{L^2 L^2} \\ \times \| \Delta_{N_2, R_2} (u_1) \| _{L^2 L^2} \sum_{ \alpha \in I}  \| \Delta_{Q_{\alpha}, R_3}(u_3)  \|_{L^2 L^2} \| \Delta_{\tilde{Q}_{\alpha},R_4 } (u_4) \|_{L^2 L^2}.
\end{multline}
Next from Cauchy-Schwarz inequality, we may write
\[
 \sum_{ \alpha \in I}  \| \Delta_{Q_{\alpha}, R_3}(u_3)  \|_{L^2 L^2} \| \Delta_{\tilde{Q}_{\alpha},R_0 } (u_0) \|_{L^2 L^2} \leq \left ( \sum_{ \alpha \in I}  \| \Delta_{Q_{\alpha}, R_3}(u_3)  \|_{L^2 L^2}^2 \right ) ^{\frac{1}{2}} \left (  \sum_{ \alpha \in I} \| \Delta_{\tilde{Q}_{\alpha},R_4 } (u_4) \|_{L^2 L^2}^2 \right ) ^{\frac{1}{2}}.
\]
First, since $(Q_{\alpha})_{\alpha}$ is a partition of $Q_3$,  by orthogonality, we have for the first term on the right hand side above:
\[
\left (  \sum_{ \alpha \in I}  \| \Delta_{Q_{\alpha}, R_3}(u_3)  \|_{L^2 L^2}^2 \right )^{ 1/2}  = \| \Delta_{Q_3, R_3} (u_3) \|_{L^2 L^2}.
\]
For the second term, the $\tilde{Q}_{\alpha} $'s recover $Q_4$ but since there are not disjoint, strict orthogonality is broken. However, using the same argument of almost orthogonality as for the proof of Strichartz estimate (each point of $Q_4$ belongs to a finite number of $\tilde{Q}_{\alpha} $), we deduce
\[
 \left (  \sum_{ \alpha \in I}  \| \Delta_{\tilde{Q}_{\alpha}, R_4}(u_4)  \|_{L^2 L^2}^2 \right )^{ 1/2} \leq C \| \Delta_{Q_4, R_4} (u_0) \|_{L^2 L^2}.
\]
Thus,
\begin{equation} \label{bourg}
\alpha(N_i, R_i) \leq C (N_1 N_2)^{1/2+ \var(b)} (R_1 R_2 R_3 R_4)^{ \beta (b)} \prod_{i=0}^3 \| \Delta_{Q_i, R_i} (u_i) \|_{L^2 L^2}. 
\end{equation}
We reorder terms to make appear Bourgain's norms of $u_i$. The quantity 
\[
 H= \sum_{\substack{ N_4 \leq 2(N_1+N_2+N_3) \\ R_1, R_2, R_3, R_4 \\ N_3 = \mathrm{Max}(N_1, N_2, N_3)}}  \alpha(N_i, R_i)
\]
is bounded by
\begin{eqnarray*}
| H| & \leq  &\sum_{N_1, R_1} N_1 ^{\frac{1}{2} + \var(b) -s} R_1^{\beta(b)-b} N_1^s R_1^b \|\Delta_{N_1, R_1} (u_1) \|_{L^2 L^2} \\
 & & \times \sum_{N_2, R_2} N_2 ^{\frac{1}{2} + \var(b) -s} R_2^{\beta(b)-b} N_1^s R_1^b \|\Delta_{N_2, R_2} (u_2) \|_{L^2 L^2} \\
 & & \times \sum_{N_4 \leq 6 N_3} \sum_{R_4, R_3} R_0^{\beta(b)-b'} R_4^{b'} R_3^{\beta (b)-b} R_3^b \| \Delta_{N_4, R_4} (u_4)\|_{L^2 L^2} \| \Delta_{N_3, R_3} (u_3)\|_{L^2 L^2}.
\end{eqnarray*}
For the first two sums above, we use Cauchy-Schwarz inequality to recover Bourgain's norm of $u_i$. For instance for the first term, we have if $s>1/2 + \var(b)$, and since $b>\beta(b)$,
\begin{eqnarray*}
\sum_{N_1, R_1} N_1 ^{\frac{1}{2} + \var(b) -s} R_1^{\beta(b)-b} N_1^s R_1^b \|\Delta_{N_1, R_1} (u_1) \|_{L^2 L^2} & \leq & \| u_1 \|_{X^{s,b}_L} \left ( \sum_{N_1, R_1} N_1 ^{1 + 2\var(b) -2s} R_1^{ 2 (\beta(b)-b)} \right )^{ \frac{1}{2}}  \\
   & \leq & C  \| u_1 \|_{X^{s,b}_L}.
\end{eqnarray*}
For the third sum, using again Cauchy-Schwarz inequality, and choosing $b' > \beta(b)$ (this condition is compatible with $1-b-b'>0$ since $\beta(b)<1-b$), we write 
\begin{eqnarray*}
\sum_{R_4} R_4^{ \beta(b)- b'}  R_4^{b'} \| \Delta_{N_4, R_4} (u_4) \|_{L^2 L^2} & \leq & \left ( \sum_{R_4} R_4^{2 \beta(b)-2b'} \right )^{\frac{1}{2}} \left ( \sum_{R_4} R_4^{2b'} \| \Delta_{N_4, R_4} (u_4) \|_{L^2 L^2} ^2 \right )^{\frac{1}{2}} \\ & \leq & C \| \Delta_{N_4} (u_4) \|_{X^{0,b'}_L}.
\end{eqnarray*}
We treat the sum over $R_3$ in the same way. Therefore, 
\[
| H |  \leq   \|u_1\|_{X^{s,b}_L} \|u_2\|_{X^{s,b}_L} \sum_{N_4 \leq 6 N_3} \frac{N_4^s}{N_3 ^s} N_4^{-s} \| \Delta_{N_4} (u_4) \|_{X^{0,b'}_L} N_3^s \| \Delta_{N_3} (u_3) \|_{X^{0,b}_L} .
\]
Now we need the following lemma (see \cite{BurGerTzv2005a} lemma~4.5 for a proof) to conclude.
\begin{lem}
For every $s>0$, there exists a constant $C>0$ such that for all sequence $(a_{N_4})_{N_4 \in 2^{\N}}, (b_{N_3})_{N_3 \in 2^{\N}}$, we have
\[
\sum_{N_4 \leq 6 N_3}  \left ( \frac{N_4}{N_3} \right )^{s} | a_{N_4} b_{N_3}|  \leq C \left ( \sum_{N_4} a_{N_4}^2 \right )^{1/2} \left ( \sum_{N_3} a_{N_3}^2 \right )^{1/2}.
\]
\end{lem}
\noindent To conclude in this case, we apply the lemma with 
\[
a_{N_4} =N_4^{-s} \| \Delta_{N_4} (u_4) \|_{X^{0,b'}_L}  , \quad b_{N_3}=N_3^s \| \Delta_{N_3} (u_3) \|_{X^{0,b}_L},
\] 
and obtain
\begin{equation} \label{bourgprop2}
 |H| \leq C \| u_1\|_{X^{s,b}_L} \| u_2  \|_{X^{s,b}_L} \| u_3\|_{X^{s,b}_L} \| u_4\|_{X^{-s,b'}_L}.
\end{equation}
\textbf{Case} \mathversion{bold} $N_3 < \mathrm{max}(N_1,N_2,N_3).$ \mathversion{normal} In the summation (\ref{summation}), we assume for instance $N_1 \leq N_3 \leq N_2$. This case is easier since we do not need to decompose high frequencies. With the definition of $\alpha(N_i, R_i)$ and from Cauchy-Schwarz inequality:
\[
 |\alpha(N_i, R_i)| \leq \| \Delta_{N_1, R_1}(u_1) \Delta_{N_2, R_2}(u_2) \|_{L^2 L^2} \| \Delta_{N_3, R_3}(u_3) \Delta_{N_4, R_4}(u_4) \|_{L^2 L^2}. 
\]
Coming back to Lemma~\ref{cube}, we have directly 
\[
\alpha(N_i, R_i) \leq (N_1 N_3)^{1/2 + \var} (R_4 R_1 R_2R_3)^{\beta(\var)} \prod_{i=1}^4 \| \Delta_{N_i, R_i}(u_i) \|_{L^2 L^2}.
\]
Once we have this estimate, the end of the proof in this case is the same as the previous one and we obtain
\begin{equation} \label{maxprop}
\sum_{ \substack{ N_3 < \mathrm{max}(N_1,N_2,N_3) \\ R_1, R_2, R_3, R_4 \\ N_4 \leq 6 \mathrm{max}(N_1,N_2,N_3) }} \alpha (N_i, R_i)  \leq C \| u_1\|_{X^{s,b}_L} \| u_2  \|_{X^{s,b}_L} \| u_3\|_{X^{s,b}_L} \| u_4\|_{X^{-s,b'}_L}.
\end{equation} 
Estimates (\ref{bourgprop2}) and (\ref{maxprop}) provides Proposition~\ref{trilinear}.
\end{proof}
Writing 
\[
\Phi(u)=e^{it P} u_0 + i \int_0^t e^{ i(t-\tau) P  }\left ( |u(\tau)|^2 u(\tau) + E( | u(\tau)|^2) u(\tau) \right ) d \tau.
\]
and using (\ref{linearestimate}), Lemma~\ref{lemmarien} and Proposition~\ref{trilinear}, we have easily
\[
 \| \Phi (u)  \|_{X^{s,b}_{L,T}} \leq C \|u_0\|_{H^s} + C T^{1-b-b'} \|u \|_{X^{s,b}_{L,T}}^3,
\]
and
\[
\| \Phi(u) - \Phi (v) \|_{X^{s,b}_{L,T}} \leq C T^{1-b-b'} \left (\| u\|_{X^{s,b}_{L,T}} ^2 + \| v\|_{X^{s,b}_{L,T}} ^2\right ) \| u-v \|_{X^{s,b}_{L,T}}.
\]
Therefore, we may close the fixed point argument in the ball $B(0,R)$ of $X^{s,b}_{L,T}$ with $R=2C \|u_0\|_{H^s}$ and $T \geq D/\|u_0\|_{H^s}^{2/(1-b-b')}$ with $D>0$ independent of the period $L \geq 1$. This proves (\ref{rate1}) for low regularities and also the first point (take $L=1$) in Theorem~\ref{newthm}.

\medskip

\textbf{Step 2.} Let us now finish the proof of the lower bound (\ref{lower_bound}). Let $u \in H^s(T^2)$ solution to (\ref{ds2}) and consider the family for $\tau \in [0, T)$:
\[
 v^{\tau}(t,x,y)= \lambda (\tau) u(\lambda ^2(\tau) t+ \tau, \lambda (\tau) x, \lambda (\tau) y), 
\]
where $\lambda (\tau)= \| u(\tau) \|_{H^s(T^2)}^{-1/s}$. For all $\tau$, $v^{\tau}$ is a function on the torus $T_{1/ \lambda (\tau)}$ and satifies the equation (\ref{eq2}) for $L=1/ \lambda (\tau)$. Moreover, it is easy to check that $\|v^{\tau}(0)\|_{L^2}=\|u(0) \|_{L^2}$ and $\| (- \Delta)^{s/2} (v^{\tau}(0))\|_{L^2} \leq 1$. If we denote by $T_{\tau}$ the maximal time for $v^{\tau}$, from  (\ref{rate1}), we deduce the uniform bound, $T_{\tau} \geq C >0$. But $T_{\tau} =(T-\tau)/\lambda ^2(\tau)$ where $T$ is the maximal time for $u$ and this with the uniform lower bound on $T_{\tau}$ proves the lower bound (\ref{lower_bound}). 

\medskip

\noindent \textbf{Acknowledgements.} I would like to thank Nikolay Tzvetkov for his guidance and helpful advices. I also thank one of the referees for the remarks which improved the presentation of the paper.

\nocite{GerPie2010}

\bibliographystyle{plain}
\bibliography{bibliography-dstorus2}

\end{document}